\documentclass[12pt]{amsart}
\usepackage{graphicx}
\usepackage{amsfonts}
\usepackage{amssymb}
\usepackage{latexsym}
\usepackage{amscd}

\newtheorem{theorem}{Theorem}[section]

\newtheorem{lemma}[theorem]{Lemma}
\newtheorem{remark}[theorem]{Remark}

\newtheorem{proposition}[theorem]{Proposition}
\theoremstyle{definition}
\newtheorem{definition}{Definition}[section]

\newenvironment{pf*}[1]{\proof[#1]}{\endproof}
\usepackage{euscript}

\usepackage[OT2,OT1]{fontenc}

\newcommand{\bbC}{\mathbb{C}}

\newcommand{\cal}[1]{{\mathcal #1}}

\theoremstyle{remark}

\newcommand{\eps}{\epsilon}

\numberwithin{equation}{section}

\newcommand{\cA}{{\cal A}}
\newcommand{\cI}{{\cal I}}

\newcommand{\cM}{{\cal M}}
\newcommand{\cV}{{\cal V}}

\newcommand{\cT}{{\cal T}}

\newcommand{\cP}{{\cal P}}

\newcommand{\cR}{{\cal R}}

\newcommand{\RR}{{\Bbb R}}

\begin{document}
\addtolength{\evensidemargin}{-0.7in}
\addtolength{\oddsidemargin}{-0.7in}

\title[Renormalization of analytic GIETs]{Renormalization and rigidity of analytic non-linear interval exchange transformations.}
\author{Goncharuk, N., Gorbovickis, I., Yampolsky, M.}
\date{\today}
\maketitle

\begin{abstract}
We construct hyperbolic renormalization horseshoes for analytic generalized interval exchange transformations (GIETs) of bounded type and with bounded geometry, provided their boundaries are small (but not necessarily zero). As a consequence, we describe the manifold structure of smooth conjugacy classes of GIETs.
  \end{abstract}

\section{Introduction}
  Renormalization technique, known as (accelerated) Rauzy-Veech induction is, of course, fundamental to the study of interval exchange transformations (IETs), see e.g. \cite{YoccozIET}. Recently, there has been much development in extending some of the results to the non-linear generalized interval exchange transformations -- GIETs (see e.g. \cite{GhU}). This is a step from understanding the action of renormalization on a finite-dimensional space of translations or, at worst, affine maps to an infinite-dimensional setting. Even in the case of GIETs of two intervals, which corresponds to maps of the circle, renormalization theory is very highly non-trivial (see e.g. the works of one of the authors \cite{Ya3,Ya4} and references therein).

  Recently, there was a beautiful paper of S.~Ghazouani \cite{Gh} in which he described the structure of the stable manifolds of renormalization of periodic IETs in the space of smooth GIETs, recovering a part of the renormalization picture. However, in the smooth setting, renormalization is not a differentiable operator. This is an old and a well-known obstacle. Ghazouani finessed it at the expense of restrictions on the combinatorial type of mappings considered -- periodic, and with a single singularity on the associated translation surface.

  In the present note, we explore a different approach, similar to our previous  works \cite{Ya3,GY,KAM-yam,GorYa} and others -- in which we define a renormalization operator $\cR$ acting on the space of analytic GIETs. Renormalization transformation $\cR$ is a compact analytic operator in a complex Banach manifold. We study its hyperbolic properties on a neighborhood of IETs of bounded type and bounded geometry. This operator is hyperbolic, and has
  finitely many expanding directions. Standard considerations of Hadamard-Perron Theorem can then be used to construct analytic stable manifolds, which coincide with smooth rigidity classes. While rigidity statements exist in the literature for wider combinatorial classes of GIETs (cf. \cite{GhU,BeTr}) our results endow rigidity classes with natural submanifold structure, similarly to \cite{Gh} but with significantly less technical effort.

  Finally, stability of hyperbolicity allows us to perturb our results to obtain similar results for projective IETs -- whose maps are M\"obius rather than affine. These results are completely novel in any setting.

\section{Interval exchange transformations}

An interval exchange transformation (IET) $T$ is a one-to-one map of the unit interval $[0,1]$ onto itself that is piecewise a translation.  Formally, we have a partition $[0,1] = \bigcup_{1\le j \le d} I_j$, and $T$ is a translation on each $I_j$.

An IET $T$ is defined by the following data: a tuple of real positive numbers  $\{\lambda_j\}_{j=1}^{d}$  with $\sum \lambda_j=1$ (lengths of the intervals $I_j$); a substitution $\sigma_t \colon \{1,\dots, d\}\to \{1, \dots, d\}$ that defines the order of the intervals before the transformation; and a substitution $\sigma_b \colon \{1,\dots, d\}\to \{1, \dots, d\}$ that defines the order of the intervals after the transformation. Here t stands for top and b stands for bottom.

In the next section, we will introduce generalized interval exchange transformations (GIETs).

\section{Space of  analytic GIETs $\cM_{\sigma, \eps}$}


As above, let $\sigma_t, \sigma_b \colon \{1,\dots, d\}\to \{1, \dots, d\}$ be two substitutions. Let $U_\eps$ be an $\eps$-neighborhood of $[0,1]$ in $\bbC$. Let $\cM_{\eps}$ be the Banach manifold of bounded analytic maps in $U_\eps$ that extend continuously to $\overline U_{\eps}$, equipped with the uniform norm.
\begin{definition}
 An analytic generalized interval exchange transformation (GIET) $T$ with substitutions $\sigma_t, \sigma_b$ is a tuple of analytic maps $f_j\in \cM_\eps$  and points $(a_0=0, a_1,\dots, a_{d-1}, a_d=1)$ in $U_\eps$ such that for $\tau=\sigma_b\sigma_t^{-1}$,

 \begin{itemize}
\item if $\tau(k)=d$,  we have  $f_k(a_k)=a_d$;

\item if $\tau(k)=1$, we have $f_k(a_{k-1})=a_0$;

\item if $\tau(k)+1=\tau(l)$, we have  $f_k (a_{k}) = f_l(a_{l-1})$.

 \end{itemize}
\end{definition}


\noindent
We note:
\begin{proposition}
Analytic GIETs with a permutation $\sigma = (\sigma_t, \sigma_b)$ form an analytic Banach manifold $\cM_{\sigma, \eps}$.
\end{proposition}

If the points $a_j$ are real and satisfy $a_0<a_1<\dots<a_d$, and $f_j$ are real-symmetric and increasing, then
we will say that $T$ is a {\it real GIET}. 
In this case, the tuple $T$ defines a piecewise-analytic bijection of $[0,1]$ in an obvious fashion, generalizing IETs to the nonlinear case.
Real GIETs carry a real Banach manifold structure from the real slice $\cM^\RR_{\sigma, \eps}$ of $\cM_{\sigma, \eps}$.

While this is the most interesting situation, we will allow complex $a_j$ in the definition, in order to obtain and use hyperbolicity of a complex-analytic renormalization operator in $\cM_{\sigma, \eps}$.

The following definition describes important classes of real GIETs.
\begin{definition}
If $T$ is a real GIET and ...
\begin{itemize}
 \item ... $f_j$ are translations, then $T$ is called the interval exchange transformation (IET) of the intervals $[a_j, a_{j+1}]$.
\item ... $f_j$ are affine maps, then $T$ is called an affine interval exchange transformation (AIET).
\item ... $f_j$ are Moebius maps, then $T$ is called a projective interval exchange transformation (PIET).
\end{itemize}
\end{definition}

\section{Rauzy induction for IETs}

\begin{definition}
 An IET $T$ satisfies the \emph{Keane condition} if the future orbit of any endpoint $a_j$ does not contain any of the points  $a_k$ for $0<k<d$.
\end{definition}
Keane condition implies minimality, see \cite{Keane}. We will see later that all IETs with Keane condition are infinitely renormalizable. In fact, all infinitely renormalizable IETs satisfy Keane condition, see \cite[Sec.3, 5]{Viana} for more details.

\subsection*{Rauzy induction}

Suppose that an IET $T$ given by a substitution $\sigma = (\sigma_t, \sigma_b)$ and lengths $(\lambda_j)$ satisfies Keane condition.
Rauzy induction is the first-return map to a  subinterval under the dynamics of $T$.

Let $I_{k_1}$, $I_{k_2}$ be the rightmost intervals on top and bottom: $\sigma_t(k_1)=d$ and $\sigma_b(k_2)=d$ . Let $A_1$ and $A_2$ be their left endpoints. Note that the Keane condition implies $\lambda_{k_1}\neq \lambda_{k_2}$.
\begin{itemize}
 \item Case 1: If $\lambda_{k_1}>\lambda_{k_2}$, the interval $I_{k_1}$ is called a winner. One step of Rauzy induction creates an IET that is a first-return map to $[0,A_2]$. New interval lengths are $\tilde \lambda_j =\lambda_j$ for $j\neq k_1 $ and $\tilde \lambda_{k_1}=\lambda_{k_1}-\lambda_{k_2}$.
\item Case 2: If $\lambda_{k_1}<\lambda_{k_2}$, the interval $I_{k_2}$ is called a winner.
One step of Rauzy induction creates an IET that is a first-return map to $[0,A_1]$. New interval lengths are  $\tilde \lambda_j =\lambda_j$ for $j\neq k_2 $ and $\tilde \lambda_{k_2}=\lambda_{k_2}-\lambda_{k_1}$.
\end{itemize}
We will not write out the formulas for the new substitutions $\sigma_t, \sigma_b$; see \cite[Sec.2]{Viana}. To obtain renormalization  $\cR T$, we perform one step of Rauzy induction and rescale the domain $[0,A_2]$ or $[0,A_1]$ respectively to the unit length.
Since Keane condition is preserved under this operation, it is sufficient for infinite renormalizability of $T$.



 Consider a \emph{Rauzy graph}, defined in the following way. Its vertices are all possible permutations $\sigma= (\sigma_t, \sigma_b)$, and its oriented edges, labeled 1 or 2, correspond to case 1 and case 2 in the Rauzy induction.

To each IET $T$ that satisfies Keane condition,  we can associate a path in the Rauzy graph that corresponds to iterates of the Rauzy induction  --- in other words,  an infinite sequence $\gamma(T)$ of edges of the Rauzy graph.
\begin{definition}
The sequence $\gamma(T)$ is called the (combinatorial) \emph{rotation number} of $T$.
\end{definition}





\section{Renormalization of analytic GIETs}


Recall that for any IET $T$ with Keane property, $\cR T$ is a rescaled first-return map and hence can be written as a tuple of rescaled compositions of restrictions $T|_{[a_{j-1}, a_j]}=f_j$. We will use the same compositions to define $\cR \tilde T$ for $\tilde T\in \cM_{\sigma, \eps}$ close to $T$.  The point $A_1$ or $A_2$ that defines the rescaling factor, as well as the new endpoints of the intervals after one step of Rauzy induction, are images or preimages of $a_j$ under $f_j$. We will use the same formulas to define the rescaling factor and new endpoints of the intervals of $\cR \tilde T$.

With this definition, the formulas that define the renormalization operator $\cR$ depend only on the first vertex and edge in the path $\gamma(T)$ in the Rauzy graph.
Since in a neighborhood of an IET $T$, all maps $f_j$ are close to translations,  we obtain the following.

\begin{lemma}
\label{lem-def}
For any rotation number $\gamma$, for any  $\eps>0$, for any IET $T$ with $\gamma(T)=\gamma$, renormalization operator $\mathcal R$ is a compact analytic operator $\mathcal R \colon \cM_{\sigma(T), \eps}\to \cM_{\sigma(\cR T), \eps}$, well-defined on a neighborhood of $T$.

Renormalization coincides with Rauzy induction on analytic real GIETs.
\end{lemma}
\begin{proof}
  The domains of definition of the analytic maps forming $\cR T$ increase after linear rescaling in the definition of $\cR$. The same holds locally in a neighborhood of $T$, which implies the local compactness of $\cR$.
    \end{proof}


\section{Invariants of renormalization: boundary and total nonlinearity}

The  classical suspension construction embeds any IET into a flow on a translation surface (see \cite[Sec.12]{Viana}). This surface has conical singularities, located on the trajectories that emanate from $\{a_j\}$. Each singularity hence corresponds to a subset of endpoints $\{a_{i_k}\}$. This subset can be defined combinatorially as follows.

Let $\tau=\sigma_b\sigma_t^{-1}$, and extend it formally by $\tau(0)=0$, $\tau(d+1)=d+1$.
\begin{definition}
\label{def-sing}
A sequence  of endpoints $\{a_{i_k}\}_{k=1}^N$ with  $0\le i_k \le d$, numbered cyclically modulo  $N$, is associated with a singularity of $T$ if, for every $1\le k \le N$, $$\tau(i_k) +1 = \tau(i_{k+1}+1).$$
\end{definition}

Let $s=s(T)$ be the number of singularities of $T$.

\begin{definition}
For each singularity of a GIET $T$ corresponding to  $\{i_k\}_{k=1}^N$, define
 $$b(\{i_k\}_{k=1}^N) = \sum_{1\le k \le N} \log  \frac{f_{i_{k}+1}'(a_{i_{k}})}{f_{i_k}'(a_{i_{k}})} $$
where we formally set $f_0=f_{d+1}=id$.

 A tuple $\bar b = (b_1,\dots, b_s)$ over all singularities of $T$, ordered according to the smallest index $i_k$ in the list $\{i_k\}$,   is called the \emph{boundary} of $T$.
\end{definition}

We note (see \cite{GhU}) that the terms of $\bar b$ can be numbered canonically according to the number of a singularity on a translation surface. This numbering is invariant under renormalization.

 Let $\cM_{\sigma, \eps}^{\bar b}$ be a set of analytic GIETs with the boundary $\bar b$.

The following lemma was proven in \cite[Lemma 2.7.1]{GhU} for real GIETs. By analyticity, it also holds for analytic GIETs.
\begin{lemma}
\label{lem-boundary}
1. For real-symmetric GIETs, the boundary is an invariant of $C^1$ conjugacy on $\RR$.

2. In assumptions of Lemma \ref{lem-def},  renormalization operator $\mathcal R$ takes $\cM_{\sigma(T), \eps}^{\bar b}$ to  $\cM_{\sigma(\cR T), \eps}^{\bar b'}$ where $\bar b'$ differs from $\bar b$ by a permutation $p$ that depends only on the rotation number of $T$.
 \end{lemma}

Furthermore, 

\begin{definition}
 Total nonlinearity of an analytic GIET is
 $$c(T) = \sum_{j=1}^{d} \int_{a_{j-1}}^{a_{j}} \frac{f_j''}{f_j'}dz =  \sum \log (f_j'(a_{j})) - \log (f_j'(a_{j-1})).$$
\end{definition}
Integrals are taken over any curves in $U_\eps$.
The second definition requires choosing continuous branches of $\log f_j'$ in $U_\eps$; since we assume that $f_j$ are close to translations, we may and will use principal branches of $\log$ from now on.

Let $\cM_{\sigma, \eps}^c$ be the set of analytic GIETs with total nonlinearity $c$. Total nonlinearity of any IET and any  AIET is zero.
By definition, singularities are in one-to-one correspondence with cycles $\{i_k\}$ of the bijection of $\{1,\dots, d\}$ given by $\tau^{-1}(\tau(i)+1)-1$. Hence the union of all sequences $\{i_k\}$ over all singularities is $\{0, 1, \dots, d\}$, therefore
$$\sum_{j=1}^s b_j = -c(T).$$
Hence, as a corollary of Lemma~\ref{lem-boundary}, the following is true:

\begin{lemma}
In the assumptions of Lemma \ref{lem-def}, renormalization operator $\mathcal R$ takes an open subset of $\cM_{\sigma(T), \eps}^c$ to  $\cM_{\sigma(\cR T), \eps}^c$.
\end{lemma}

\section{Dynamics on AIETs and the statement of the main theorem}

Renormalization operator preserves the class of AIETs. Here we describe its action on AIETs, cf. \cite[Sec.3.2]{Gh}.

Let $\cV$ be the set of vertices of the Rauzy graph.
An AIET is defined by the substitution $\sigma\in \cV$, length vector $\bar \lambda = \{\lambda_j\}$ with
\begin{equation}
\label{eq-sum}
\sum_{j=1}^d \lambda_j=1,
\end{equation}
 and log-slopes of affine maps $f_j$ on the intervals $[a_j, a_{j+1}]$, given by $\rho_j=\log f_j'$ and satisfying
 \begin{equation}
\label{eq-sum-w}
\sum_{j=1}^d \exp(\rho_j)\lambda_j=1,
\end{equation}
 Hence the space $\cA_d$ of real AIETs can be identified with a codimension-1 subset of $\cV\times P\RR^d \times \RR^d$ where $P$ stands for projectivization.
Let $\cT_d\subset \mathcal A_d$ be the subset of IETs inside this space.

Let us fix an invariant (under $\cR$) compact set $\Lambda \subset \cT_d$. Let $T\in\Lambda$. Compactness implies that the lengths of the intervals in the forward orbit of $T$ under renormalization are detached from zero.

For one step of Rauzy induction on AIETs, formulas for the lengths change to $\tilde \lambda_{k_1} = \lambda_{k_1}-\exp(\rho_{k_2})\lambda_{k_2}$ in case 1 and $\tilde \lambda_{k_2} = \lambda_{k_2}-\exp(\rho_{k_1})\lambda_{k_1}$ in case 2. These should be normalized to the unit length to obtain the length vector for $\cR T$.  On the slopes, one step of renormalization acts by applying a linear operator $A=I+\delta_{k_1k_2}$ or $A=I+\delta_{k_2k_1}$ in case 1 and 2 respectively to the vector  $\bar \rho = \{ \rho_j\}$. We will write $A=A_{\gamma(T), 1}$ and remark that this matrix only depends on the first vertex and edge in $\gamma(T)$.

Under iterates of renormalization, the action on the slopes is given by matrices $$A_{\gamma(T),n} =A_{\gamma(\cR^n T), 1} \circ \dots \circ A_{\gamma(T), 1}.$$ These were called intersection matrices in \cite{Gh}: an $(i,j)$-term of such matrix can be interpreted as the number of visits of the $j$-th interval of  $\cR^n T$ to the $i$-th interval of $T$ under the iterates of $T$ that occur before the first return.

It is well-known (\cite[Corollary 1.21]{Viana}) that for any IET $T$, for some $N$ and any $n>N$, the matrices $A_{\gamma(T),n}$ have strictly positive terms. Note that in our case, $N$ is uniformly bounded by compactness of $\Lambda$.

The action of renormalization on AIETs is a linear cocycle over the action of $A_{\gamma(T),n}$.
Moreover,  the above formulas for lengths and slopes imply that the differential of renormalization $\mathcal R^n$ in the space $\cA_d$ at any point $T\in \cI_d$ has the form $$B_{\gamma(T), n}= \begin{pmatrix}P({}A_{\gamma(T),n}^t)^{-1} & 0 \\ * & A_{\gamma(T),n}\end{pmatrix}$$ restricted to the linear subspaces \eqref{eq-sum}, \eqref{eq-sum-w},  where $P(A_{\gamma(T),n}^t)^{-1}$ is the projectivization (to account for rescalings) of the matrices  $(A_{\gamma(T),n}^t)^{-1}$.

The restriction of  $\cR$ to IETs is  uniformly expanding at $T\in \Lambda$. This is straightforward to see: since matrices $A_{\gamma(T),n}$, $T\in \Lambda$ have positive terms for some uniformly bounded $n$, they uniformly contract the simplex of positive vectors and hence their projectivizations exponentially contract at a uniform rate. Therefore   $P(A_{\gamma(T),n}^t)^{-1}$ uniformly expand (see \cite[Remark 4.26]{Viana}).
The analogy is with the action of the Gauss map on a compact invariant set of irrationals of bounded type.




Since renormalization preserves the boundary, matrices $B_{\gamma(T), n}$ have an invariant  distribution $V_{\gamma(T), n}$ in the tangent bundle to $\cA_d $ of codimension $s=s(T)$, given by the condition that the boundary is zero,  with neutral dynamics in the transversal direction. Since this distribution is defined by the restrictions on the slopes only, we can talk about the dynamics of $A_{\gamma(T), k}$ in this distribution.

In the Main Theorem below, we will assume that this dynamics is uniformly hyperbolic with $l$ unstable directions. The dynamics of $B_{\gamma(T), n}$ will then have $l+(d-1)$ unstable and $s$ neutral directions.

Recall that $\sigma_n$ are substitutions corresponding to $\cR^n T$  (note that there are only finitely many different substitutions in this list). Recall that $\cM_{\sigma_n, \eps}$ are the corresponding Banach manifolds of GIETs, and $\cM_{\sigma_n, \eps}^{\bar b}\subset \cM_{\sigma_n, \eps}$ are codimension-$s$ subspaces of GIETs with boundary $\bar b$. Let $\cM_{\sigma, \eps}^\RR$, $(\cM_{\sigma, \eps}^{\bar b})^\RR$ be real slices of these spaces.

\begin{theorem}[Main Theorem]
  \label{thm-main1}
  For any invariant compact set $\Lambda \subset \cT_d$, let $T\in\Lambda$.
Assume that the  associated action of the sequence of intersection matrices $A_{\gamma(T), n}$ is uniformly hyperbolic with $l$ unstable directions in restriction to  $V_{\gamma(T), n}$.
Then the following holds:

\begin{enumerate}
\item The orbit of $T$ is hyperbolic with $l+(d-1)$ unstable directions under renormalization in the spaces $\cM_{\sigma_n, \eps}^{\bar 0}$.

\item For any $\eps>0$, for some $\kappa>0$, there exists a  local real-analytic submanifold of $\cM_{\sigma, \eps}^{\bar 0}$ at $T$ of codimension $l+(d-1)$ that is  formed by real GIETs that are $C^{1+\kappa}$-conjugate to $T$. This submanifold coincides with the local stable submanifold of $T$ under $\cR$ in $(\cM_{\sigma, \eps}^{\bar 0})^\RR.$

\item Moreover, for any real $\bar b$ sufficiently close to $\bar 0$, there exists a non-empty local   real-analytic manifold in $\cM_{\sigma, \eps}^{\bar b}$ of codimension $l+(d-1)$ formed by mutually $C^{1+\kappa}$-conjugate real GIETs with boundary $\bar b$. These submanifolds coincide with local stable manifolds of $\cR$ in $(\cM_{\sigma, \eps}^{\bar b})^\RR.$

\item Finally, for non real-symmetric GIETs which  belong to the same local stable submanifold of $\cR$ in the complexified space
$\cM_{\eps,\sigma}^{\bar b}$ with $\bar b\approx \bar 0$ as above, we have the following. They possess invariant quasisymmetric curves with endpoints $0,1$ in $\bbC$ the dynamics on which are quasiconformally conjugate.
\end{enumerate}
  
\end{theorem}
\begin{remark}[Computation of codimension]
 Since the matrices $A_{\gamma(T), n}$ preserve a symplectic form (see \cite[Proposition 5.20]{Viana}), hyperbolicity in a codimension-$s$ distribution inside a $d-1$-dimensional space of slope vectors implies $l=(d-s-1)/2$.

 We can express codimension in terms of the genus of the translation surface, given by  $2g=d-s+1$. Namely, we get $l=g-1$ and hence the codimensions in Main Theorem become $(d-1)+(g-1)$ (inside the space of GIETs with fixed boundary) and $(d-1)+(g-1)+s$ (in the ambient space). This agrees with Problem 1 in \cite{MMY-int}, and with the result of \cite{Gh} that is proved for $s=1$.

\end{remark}

\section{Proof of the Main Theorem}

\subsection{Dynamics on PIETs}

Note that PIETs form an invariant finite-dimensional sequence of local manifolds $\cP_n \subset \cM_{\sigma_n, \eps}^0$ at the orbit $\{\cR^n T\}$. Let $p_n$ be permutations from Lemma \ref{lem-boundary} associated with the $n$-th renormalization of  $T$, note that they only depend on the rotation number of $T$.  Assume $\bar b_n=p_n(\bar b)$ and let $\cP_n^{\bar b_n} \subset \cM_{\sigma_n,\eps}^{\bar b_n}$ be a submanifold of PIETs. Since renormalization of a Moebius IET is again Moebius, we have locally near an IET $T$ that $$\cR (\cP_n^{\bar b_n}) \subset \cP_{n+1}^{\bar b_{n+1}}.$$

\begin{proposition}
  \label{prop-piet}
  \begin{itemize}
\item[(I)]In assumptions of the Main Theorem, the orbit of $T$ within $\cP_n^{\bar 0}$ is hyperbolic, with $l+(d-1)$ unstable directions.

\item[(II)] For any sufficiently small $\|\bar b\|$, there exists a PIET $T_{\bar b}\in\cP_0^{\bar b}$ such that its orbit within  $\cP_n^{\bar b_n}$ is hyperbolic, with $l+(d-1)$ unstable directions.
\end{itemize}

  \end{proposition}
\begin{proof}
The first statement is proved in \cite[Proposition 20]{Gh}. Namely, since the total nonlinearity of the maps in $\cP_n^{\bar 0}$ is zero, PIETs converge to AIETs.
Since the boundary is preserved,  these AIETs will also have zero boundary. Inside the distributions $V_{\gamma(T), n}$ that correspond to  AIETs with zero boundary, renormalization operator has  $l+(d-1)$ unstable directions by the assumptions.

Second statement holds true due to structural stability of hyperbolicity in finite-dimensional manifolds.
Formally, we can easily construct a projection from $\cP_n^{\bar b} $ to $\cP_n^{\bar 0}$ that is one-to-one and close to identity, e.g. by post-composing $f_j$ with affine maps, and hence we can treat $\cR|_{\cP_n^{\bar b} }$ as a small perturbation of $\cR|_{\cP_n^{\bar 0} }$.
\end{proof}

\subsection{Extending hyperbolicity to  analytic GIETs}
The standard technique for showing that analytic real GIETs converge to PIETs is to use the Schwartzian derivatives (cf. \cite{GhU}).
Namely, suppose that $T$ is an IET satisfying assumptions of the Main Theorem, and let $\tilde T$ be a real GIET that is $n$ times renormalizable so that  its renormalizations stay in small neighborhoods of $\cR^i T$ for $i\leq n$. Let $f_j^n$ be the maps and $a_j^n$ be the points that correspond to $\cR^n(\tilde T)$. In \cite[Proposition 4.4.1, 4.4.2]{GhU}, the authors prove an exponential decay of Schwartzians:
$$\sup_{[a_j^n, a_j^{n+1}]} S(f_j^n)  < c(T) |P_n| \sup_{[0,1]} S(T)$$
where $|P_n|$ is the maximal size of the intervals in the $n$-th dynamical partition associated with $\cR^n T$ (cf. \cite[Proposition 4.4.1, 4.4.2]{GhU}). Since we consider $T\in \Lambda$ for a compact invariant set  $\Lambda$,  the latter decays exponentially.


Assuming that $\tilde T$ is sufficiently close to $T$, the Schwarzians   $S(f_j^n)$ are well-defined analytic functions on a fixed neighborhood $U_{\eps'}[0,1]\Supset U_\eps([0,1])$ since renormalization improves the domain.
Hence their norms decrease exponentially and uniformly in all of $U_\eps([0,1])$. Thus,
the linear functionals $DS(f_j)$ are contracted exponentially, hence $D\cR$ exponentially contracts towards PIETs.

This proves part (1) of the Main Theorem. Furthermore, in view of Proposition~\ref{prop-piet}, the orbit of PIET $T_{\bar b}\in\cP_0^{\bar b}$ is hyperbolic with $l+(d-1)$ unstable directions in $\cM_{\sigma, \eps}^{\bar b}$. The analyticity of the stable manifolds of $T=T_{\bar 0}$ and $T_{\bar b}$
follows from the infinite-dimensional version of the Hadamard-Perron Theorem \cite{elbialy1}.
The $C^{1+\kappa}$ conjugacy follows from geometric convergence of  renormalization in this case in the standard fashion.
This proves parts (2), (3) of the Main Theorem.

The proof of part (4) of the Main Theorem is similar to the argument used in \cite{Ya5,GY}. The local stable manifold of renormalization
$$W=W^s_{\text{loc}}(T_{\bar b})\subset \cM_{\sigma, \eps}^{\bar b}$$
is a complex analytic manifold. The orbits under the GIETs of the endpoints $0$, $1$ move holomorphically over $W$. Applying  the infinite-dimensional Lambda-lemma  \cite{BR}, we see that their closures move holomorphically as well, and are quasiarcs, the dynamics on which is quasiconformally conjugate to that of $T_{\bar b}$.








\bibliographystyle{amsalpha}
\bibliography{biblio}

\end{document}